\documentclass{amsart}
\usepackage{amsthm}
\usepackage{amsmath,amscd,amssymb}
\usepackage{enumerate}
\usepackage{hyperref}
\usepackage{graphicx}
\usepackage{pinlabel}
\usepackage{xypic}
\usepackage{enumitem}

\newcommand{\RR}{\mathbb{R}}
\newcommand{\QQ}{\mathbb{Q}}

\newcommand{\ZZ}{\mathbb{Z}}
\newcommand{\NN}{\mathbb{N}}
\newcommand{\FF}{\mathbb{F}}

\newcommand{\kk}{\mathbf{k}}

\newcommand{\OP}{\operatorname}

\newtheorem{theorem}{Theorem}[section] 
\newtheorem{lemma}[theorem]{Lemma}     
\newtheorem{corollary}[theorem]{Corollary}
\newtheorem{proposition}[theorem]{Proposition}

\theoremstyle{definition}
\newtheorem{definition}[theorem]{Definition}
\newtheorem{example}[theorem]{Example}
\theoremstyle{remark}
\newtheorem{remark}[theorem]{Remark}

\title[Nontriviality of the characteristic algebra]
 {Nontriviality results for the characteristic algebra of a DGA} 

\author{Georgios Dimitroglou Rizell}

\thanks{The author is supported by the grant KAW 2013.0321 from the Knut and Alice Wallenberg Foundation. Part of this work was done during a visit of the author to the Institut Mittag-Leffler (Djursholm, Sweden).}

\begin{document}

\maketitle

\begin{abstract}
Assume that we are given a semifree noncommutative differential graded algebra (DGA for short) whose differential respects an action filtration. We show that the canonical unital algebra map from the homology of the DGA to its characteristic algebra, i.e.~the quotient of the underlying algebra by the two-sided ideal generated by the boundaries, is a monomorphism. The main tool that we use is the weak division algorithm in free noncommutative algebras due to P. Cohn.
\end{abstract}

\section{Introduction}
Differential graded algebras (DGAs for short) appear in algebraic topology. For example, they appear in the formulation of rational homotopy theory in \cite{Sullivan:Infinitesimal} due to D. Sullivan. They also appear in modern symplectic and contact topology, such as in Legendrian contact homology by Y. Chekanov \cite{DiffAlg} as well as in the more general theory of symplectic field theory by Y. Eliashberg, H. Hofer, and A. Givental \cite{IntroSFT}. We will be interested in DGAs that naturally appear in the latter setting, of which we give a very rough outline below in Section \ref{sec:contact}. This is also the setting in which all applications known to the author can be found. However, we emphasise that the results in this paper are purely algebraic.

We will be mainly interested in DGAs that are finitely generated and semifree, and thus in particular fully noncommutative, as they appear in the geometric context of Legendrian contact homology. These DGAs will also be equipped with an action filtration which is respected by the differential, which naturally appears in the latter context. We refer to Section \ref{sec:setup} for the precise algebraic definitions.

Although easily described, it is in general not an easy problem to distinguish two DGAs. In the geometric context of Legendrian contact homology considered here, one is in particular interested in the stable-tame isomorphism class of the DGA (see Section \ref{sec:stabletame}). In \cite{Computable} L. Ng introduced the characteristic algebra as a tool to study DGAs under this relation. The question that we will given an answer to here is: to what extent does the characteristic algebra remember the homology algebra of the DGA? In addition, we briefly discuss acyclic DGAs in Section \ref{sec:acyclic}.

\subsection{Main results}
In the following we consider a finitely generated semifree, thus fully noncommutative, DGA $(\mathcal{A},\partial)$ over a field $\kk$ which is of the form considered in Section \ref{sec:setup}, together with its associated characteristic algebra $\mathcal{C}:=\mathcal{A}/\mathcal{A}\partial(\mathcal{A})\mathcal{A}$. In particular we assume that the differential $\partial$ respects an action filtration as postulated by condition \ref{F} in the same section. For short, we say that $(\mathcal{A},\partial)$ is a {\bf DGA with action filtration}. Our main result in this setting is as follows.
\begin{theorem}
\label{thm:main}
Consider a DGA $(\mathcal{A},\partial)$ with action filtration as above, whose corresponding characteristic algebra is denoted by $\mathcal{C}$. The natural unital algebra morphism
\[H(\mathcal{A},\partial) \to \mathcal{C}\]
induced by the inclusion of the cycles is a monomorphism.
\end{theorem}
The proof of the main theorem relies heavily on the fact that the DGA has an action filtration respected by the differential. It is not clear to the author if this condition can be omitted.

As an immediate consequence, we obtain the following useful result.
\begin{corollary}
\label{cor:nontrivial}
A DGA with action filtration is acyclic if and only if its characteristic algebra is trivial.
\end{corollary}
We note that, in its full generality, \cite[Theorem 1.6]{MyCaps} by the author depends on this result, which was stated as \cite[Lemma 3.1]{MyCaps} in the same article, but appeared there without a proof.

\subsection{Results in the super-commutative case}
In the following we will investigate what can be said in the case of a {\bf super-commutative} DGA. In other words, a DGA as defined in Section \ref{sec:setup}, but where we have imposed the commutativity relation
\[ a\cdot b=(-1)^{|a||b|}b \cdot a.\]
Clearly, any element $b$ of odd degree satisfies $b^2=0$ unless the ground field $\kk$ is of characteristic two.

In general, the statement analogous to Theorem \ref{thm:main} is not satisfied for a super-commutative DGA; see Example \ref{ex:comm} below. However, it is not difficult to establish the following result which is analogous to Corollary \ref{cor:nontrivial}.
\begin{proposition}
\label{prp:comm}
Let $(\mathcal{A},\partial)$ be a super-commutative DGA over the ground field $\kk$ for which either:
\begin{enumerate}
\item The field satisfies $\OP{char} \kk \neq 2$, and the grading is taken in the group $\ZZ/2\ZZ$; or
\item The field satisfies $\OP{char} \kk = 2$, and the grading is arbitrary.
\end{enumerate}
It follows that the characteristic algebra $\mathcal{A}/\mathcal{A}\partial(\mathcal{A})\mathcal{A}$ is trivial if and only if the DGA is acyclic.

In the first case, the commutative unital algebra defined as the quotient $\mathcal{C}/\mathcal{C}^{\OP{odd}}$, where $\mathcal{C}^{\OP{odd}} \subset \mathcal{C}$ denotes the two-sided ideal generated by the elements in odd degree is, moreover, trivial if and only if the DGA is acyclic.
\end{proposition}
\begin{proof}
Suppose that the characteristic algebra is trivial, which implies that
\[ 1=x_1 \partial(y_1) + \hdots + x_n\partial(y_n)\]
for elements $x_i,y_i \in \mathcal{A}$, $i=1,\hdots,n$, which all may be assumed to be of homogeneous degrees.

(1): Recall that the differential $\partial$ is of degree $-1$, that $|1|=0$, and that
\[ (-1)^{|x_i|}x_i\partial(y_i)+ (-1)^{|y_i|(|x_i|-1)}y_i\partial(x_i)=\partial(x_iy_i).\]
Using these relations, may write
\[ 1=u_1 \partial(v_1) + \hdots + u_n\partial(v_n)+\partial(w),\]
for elements $u_i,w_i \in \mathcal{A}$ satisfying
\[|u_i|=|\partial(w_i)|=1 \in \ZZ/2\ZZ, \:\:i=1,\hdots,n.\]
The equation
\[ 1-\partial(w)=u_1 \partial(v_1) + \hdots + u_n\partial(v_n)\]
implies that
\[ (1-\partial(w))^2=(u_1 \partial(v_1) + \hdots + u_n\partial(v_n))^2=0\]
is satisfied by degree reasons, from which it follows that
\[ 1=\partial(2w-w\partial(w)).\]
Hence, $(\mathcal{A},\partial)$ is acyclic as sought.

In the same manner, we also see that $1 \in \mathcal{A}$ is contained in the two-sided ideal generated by the odd elements if and only if $(\mathcal{A},\partial)$ is acyclic.

(2): In the case when $\OP{char} \kk=2$ we can square both sides of the first equation above, giving rise to the relation
\[ 1=(x_1 \partial(y_1) + \hdots + x_n\partial(y_n))^2=x_1^2\partial(y_1)^2+\hdots+x_n^2\partial(y_n)^2.\]
Since every square is a cycle in this characteristic, i.e.
\[\partial(x^2)=x\partial(x)+\partial(x)x=0\]
by commutativity, the Leibniz rule again implies that
\[ 1=\partial(x_1^2y_1\partial(y_1))+\hdots+\partial(x_n^2y_n\partial(y_n))\]
is a boundary.
\end{proof}

\begin{corollary}
Given that $(\mathcal{A},\partial)$ is of one of the forms as prescribed by Proposition \ref{prp:comm}, it follows that there exists an augmentation into a field $\FF \supset \kk$, i.e.~a unital DGA morphism
\[ \varepsilon \colon (\mathcal{A},\partial) \to (\FF,0) \]
considered as a unital DGA with an empty generating set, if and only if $(\mathcal{A},\partial)$ is not acyclic.
\end{corollary}
\begin{proof} The existence of an augmentation is equivalent to the existence of a unital algebra map $\mathcal{C} \to \FF$ from the characteristic algebra to a field. Such a map exists since $\mathcal{C}$ admits a unital algebra map to a \emph{commutative} unital algebra in both of the cases covered by Proposition \ref{prp:comm}.
\end{proof}

\begin{example}
\label{ex:comm}
Consider the super-commutative DGA $(\mathcal{A}=\langle b,b_1,b_2,c \rangle,\partial)$ over $\QQ$ generated by one generator $c$ of degree $|c|=-1$, and three  generators $b_1,b_2,b$ of degree $|b_1|=|b_2|=|b|=1$. We prescribe the relations
\begin{eqnarray*}
& & \partial(b_1)=bc,\\
& & \partial(b_2)=bc,
\end{eqnarray*}
while the other generators are cycles. It immediately follows that $\partial^2=0$. Moreover, $\partial(b_1)b_2=bcb_2$ is a cycle which is not a boundary, but whose image inside the characteristic algebra $\mathcal{C}=\mathcal{A}/\mathcal{A}\partial(\mathcal{A})\mathcal{A}$ clearly vanishes.

In order to see that $bcb_2$ is not a boundary, we argue as follows. First, observe that any word of length $5$ or more automatically vanishes in this DGA. Second, we consider the computations
\begin{eqnarray*}
& & \partial(b_1b_2)=bcb_2-b_1bc,\\
& & \partial(b_ib_i)=0,\\
& & \partial(b_ib)=0,\\
& & \partial(b_ic)=0,\\
& & \partial(bc)=0,\\
& & \partial(b_ib_jb)=0,\\
& & \partial(b_ib_jc)=0,\\
& & \partial(b_ib_jbc)=0,
\end{eqnarray*}
for any $i,j \in \{1,2\}$.
\end{example}

\subsection{A brief introduction to the geometric setup in which our DGAs arise}
\label{sec:contact}
The Chekanov-Eliashberg algebra associated to a Legendrian knot of a contact manifold, as introduced independently in \cite{DiffAlg} by Chekanov and \cite{IntroSFT} by Eliashberg-Givental-Hofer, is in the basic geometric setup a natural example of a finitely generated semifree DGA. Given a Legendrian submanifold, the theory associates to it the so-called Chekanov-Eliashberg algebra $(\mathcal{A},\partial)$ over a field $\kk$. The differential $\partial$ is defined by a count of associated rigid pseudoholomorphic polygons. The homotopy type of this DGA has been shown to be a powerful Legendrian isotopy invariant. An important algebraic feature of this DGA, due to this geometric setup, is that the differential respects an action filtration. This will later turn out to be important.

The basic case where the technical details of Legendrian contact homology has been carried out is that for a Legendrian submanifold $\Lambda \subset (P \times \RR,dz+\theta)$ of a contactisation of a $2n$-dimensional Liouville manifold $(P,d\theta)$; see \cite{DiffAlg} and \cite{ContHomP}. A Liouville manifold is a particular exact symplectic manifold $(P,d\theta)$, where the latter is a pair consisting of a non-compact smooth $2n$-dimensional manifold $P$ together with an exact two-form $d\theta$ satisfying the property that $d\theta^{\wedge n}$ is a volume form on $P$. Recall that a submanifold of a $(2n+1)$-dimensional contactisation as above is Legendrian given that it is of dimension $n$, and that the contact one-form $dz+\theta$ vanishes along it.
\begin{example} The standard contact $(2n+1)$-space $(\RR^{2n+1},dz-(y_1dx_1 + \hdots + y_ndx_n))$ is the archetypal example, as well as a local model, of a contact manifold.
\end{example}
The study of Legendrian submanifolds up to Legendrian isotopy, i.e.~smooth isotopy through Legendrian submanifolds, has been shown to be a both subtle and rich field; see e.g.~\cite{DiffAlg}, \cite{Computable}, \cite{NonIsoLeg}, and \cite{Sivek:ContactHomology}, among others. The Chekanov-Eliashberg algebra is the main invariant used, which is considerably more powerful than classical topological invariants such as the rotation number and Thurston-Bennequin invariant. More precisely, for a Legendrian submanifold of a contact manifold as above, the DGA-homotopy type, and even the so-called stable-tame isomorphism type, of the Chekanov-Eliashberg algebra is invariant under Legendrian isotopy \cite{ContHomP}.

Since a semifree  DGA typically is an infinite-dimensional noncommutative algebra, it is in general not easily studied. For that reason one usually tries to derive finite-dimensional invariants from it. For instance, given an augmentation of the DGA, i.e.~a unital DGA morphism $(\mathcal{A},\partial) \to (\kk,0)$, one can use Chekanov's linearisation procedure in \cite{DiffAlg} to produce a finite-dimensional complex. The set of isomorphism classes of the homologies of all linearisations is an invariant of the DGA up to DGA homotopy.

The above linearisation procedure produces computable Legendrian isotopy invariants. However, far from all interesting Legendrian submanifolds have Chekanov-Eliashberg algebras admitting augmentations. On one hand, as shown by Proposition \ref{prp:acyclic}, there is a unique acyclic Chekanov-Eliashberg algebra up to stable tame isomorphism. On the other hand, being non-acyclic is a necessary but clearly not a sufficient condition for admitting an augmentation.

In \cite{Computable} NG introduced the so-called {\bf characteristic algebra} associated to a DGA $(\mathcal{A},\partial)$, namely the quotient $\mathcal{C}:=\mathcal{A}/\mathcal{A}\partial(\mathcal{A})\mathcal{A}$ under the two-sided ideal generated by the boundaries. In the same article, he also successfully used this algebra in order to obtain invariants from the DGA in cases when there are no augmentations.

Augmentations clearly factorise through the characteristic algebra, which can be considered as the ``universal'' (not necessarily commutative) augmentation $(\mathcal{A},\partial) \to (\mathcal{C},0)$. The first important question that now arises is under what conditions the characteristic algebra is nontrivial, which is what we study here.

For more details and applications concerning augmentations in the characteristic algebra, we refer to \cite{MyCaps} and \cite{EstimNumbrReebChordLinReprCharAlg} due to the author and the author together with R. Golovko, respectively. We also refer to \cite{Sivek:ContactHomology} for computations of Chekanov-Eliashberg algebras of interesting Legendrian knots due to S. Sivek. For instance, examples of Legendrian knots are produced for which the characteristic algebra admits a two-dimensional but no one-dimensional representation, as well as knots for which the characteristic algebra admits no finite-dimensional representations.

\section{Results concerning acyclic DGAs}
\label{sec:acyclic}

First we show the basic result that the stable-tame isomorphism class of an acyclic DGA does not contain any interesting information.
\begin{proposition} \label{prp:acyclic}
Consider two DGAs, where each DGA has a generator whose boundary is equal to the unit. Two such acyclic DGAs are tame isomorphic if and only if there is a degree-preserving bijection between their generators.
\end{proposition}
\begin{proof}
Let $b$ be a generator for which $\partial(b)=1$, and let $a$ be any other generator. The elementary automorphism $\Phi$ defined by $a \mapsto a + b\partial(a)$ satisfies the property that $\Phi \circ \partial \circ \Phi^{-1}(a)=0$. After applying a suitable tame isomorphism, we may thus assume that all generators except $b$ are cycles in both DGAs considered.
\end{proof}
\begin{corollary}
Two acyclic DGAs $(\mathcal{A}_i=\langle a_1,\hdots,a_{k_i},b_1,\hdots,b_{l_i}\rangle,\partial_i)$, $i=0,1$, with generators $a_j$ and $b_j$ in even and odd degrees, respectively, are stable-tame isomorphic if and only if $l_1-k_1=l_0-k_0$.
\end{corollary}
\begin{proof} Consider an acyclic DGA $(\mathcal{A},\partial)$ and choose an element $x \in \mathcal{A}$ of degree $1$ satisfying $\partial(x)=1$. After taking the free product with a stabilisation $(\mathcal{S}_1,\partial_{\mathcal{S}_1})$ in degree $1$ (see Section \ref{sec:stabletame}), we find generators $a$ and $b$ of degrees $1$ and $2$, respectively, for which $\partial(a)=0$. Observe that such a free product does not affect the difference between the number of odd and even generators. After the elementary automorphism $\Phi$ determined by $a \mapsto a-x$, the differential satisfies $1=\Phi \circ \partial \circ \Phi^{-1}(a)$ for the new generator $a$.

Applying this argument to the DGAs $(\mathcal{A}_i,\partial_i)$, $i=0,1$, we may assume that they both have a generator whose boundary is equal to $1 \in \mathcal{A}_i$. The above proposition combined with Lemma \ref{lem:stableiso} finally implies the existence of the sought stable-tame isomorphism.
\end{proof}
The following lemma is standard.
\begin{lemma}
\label{lem:stableiso}
Two DGAs $(\mathcal{A}_i=\langle a_1,\hdots,a_{k_i},b_1,\hdots,b_{l_i}\rangle,\partial_i)$, $i=0,1$, with generators $a_j$ and $b_j$ in even and odd degrees, respectively, are stable-tame isomorphic \emph{as graded algebras} (i.e.~forgetting the differential) if and only if $l_1-k_1=l_0-k_0$.
\end{lemma}

On the other hand, there are examples of acyclic DGAs that are not isomorphic, but whose generators can be identified by a bijection which preserves the grading. We produce such an example just out of curiosity, and note that it is irrelevant for the application of invariants of Legendrian submanifolds. Namely, in this setting it is the \emph{stable}-tame isomorphism class of the DGA that contains invariant information.
\begin{example}
\label{ex:acyclic}
Consider the following two DGAs whose underlying graded algebra is given by $\mathcal{A}=\langle a_1,a_2,b_1,b_2 \rangle$, where the elements $a_1,a_2$ are of degree $|a_1|=|a_2|=0$, while the elements $b_1,b_2$ are of degree $|b_1|=|b_2|=1$.
\begin{enumerate}
\item The first DGA has differential $\partial_0$ defined by prescribing
\begin{eqnarray*}
\partial_0(a_1) &=& 0,\\
\partial_0(a_2) &=& 0,\\
\partial_0(b_1) &=& 1,\\
\partial_0(b_2) &=& 0,
\end{eqnarray*}
which clearly yields an acyclic DGA.
\item The second DGA has differential $\partial$ defined by prescribing
\begin{eqnarray*}
\partial(a_1) &=& 0,\\
\partial(a_2) &=& 0,\\
\partial(b_1) &=& 1+a_1a_2,\\
\partial(b_2) &=& a_1^2.
\end{eqnarray*}
Since
\[ \partial(b_1-a_1b_1a_2+b_2a_2^2)=1,\]
it follows that this DGA is acyclic as well.
\end{enumerate}
\end{example}
\begin{proposition} The two semifree DGAs described in Example \ref{ex:acyclic} are not isomorphic.
\end{proposition}
\begin{proof}
The first DGA has the property that the elements in $\partial_0^{-1}(1) \subset \mathcal{A}$ of homogeneous degree $1$ all can be written as
\[ x=b_1+\sum_{i=1}^m u_ib_2v_i+\sum_{i=1}^n x_ib_1y_i,\]
for arbitrary elements $u_i,v_i \in \langle a_1,a_2 \rangle \subset \mathcal{A}$, $i=1,\hdots,m$, together with elements $x_i,y_i \in \langle a_1,a_2\rangle$, $i=1,\hdots,n$, satisfying $\sum_{i=1}^m x_iy_i=0$. Consider the quotient projection
\begin{gather*}
\phi_{(r_1,r_2)} \colon \mathcal{A} \to \mathcal{A}/\mathcal{A}R_{(r_1,r_2)}\mathcal{A}, \\
R_{(r_1,r_2)}:=\{a_1-r_1,a_2-r_2\} \subset \mathcal{A}, \:\: (r_1,r_2) \in \kk^2,
\end{gather*}
of algebras together with the canonical identification $\psi_{(r_1,r_2)} \colon \mathcal{A}/\mathcal{A}R_{(r_1,r_2)}\mathcal{A} \xrightarrow{\simeq} \langle b_1,b_2 \rangle$.

It follows that for \emph{any} pair $(r_1,r_2) \in \kk^2$ and element $x \in \partial_0^{-1}(1)$, the subset $\{\psi_{(r_1,r_2)} \circ \phi_{(r_1,r_2)}(x),b_2\} \subset \langle b_1,b_2 \rangle$ is a free generating set of the algebra $\langle b_1,b_2 \rangle$. Namely,
\[\psi_{(r_1,r_2)} \circ \phi_{(r_1,r_2)}(x)=\psi_{(r_1,r_2)} \circ \phi_{(r_1,r_2)}\left(b_1+\sum_{i=1}^mu_iv_ib_2\right)\]
by the above.

On the other hand, for any fixed element $s_1b_1+s_2b_2 \in \langle b_1,b_2 \rangle$, $s_i \in \kk$, of degree 1, we can readily find a pair $(r_1,r_2) \in \kk^2$ for which
\[\{\psi_{(r_1,r_2)} \circ \phi_{(r_1,r_2)}(b_1-a_1b_1a_2+b_2a_2^2),s_1b_1+s_2b_2\} \subset \langle b_1,b_2 \rangle\]
is linearly dependent over $\kk$ (and hence, in particular does not generate $\langle b_1,b_2 \rangle$). Since $b_1-a_1b_1a_2+b_2a_2^2 \in \partial^{-1}(1)$, it follows that the two DGAs cannot be isomorphic.
\end{proof}

\section{The algebraic setup}
\label{sec:setup}

In this section we give a precise definition of the algebraic setup used. In the following $\kk$ will denote an arbitrary field. By $\langle a_1,\hdots,a_k \rangle$ we will denote the free, fully noncommutative, associative, and unital $\kk$-algebra generated by elements $a_1,\hdots,a_k$. For a subset $B \subset \mathcal{A}$, we use $\mathcal{A}B$, $B\mathcal{A}$, and $\mathcal{A}B\mathcal{A}$ to denote the corresponding left, right, and two-sided ideals generated by $B$, respectively. The free product of two such algebras $\langle a_1,\hdots,a_k \rangle$ and $\langle a_{k+1},\hdots,a_{k+l}\rangle$ is the free algebra $\langle a_1,\hdots,a_k,a_{k+1},\hdots,a_{k+l} \rangle$.

\subsection{Semifree DGAs}
The main object of interest will be a so-called {\bf semifree differential graded algebra} (DGA for short) $(\mathcal{A},\partial)$ over a field $\kk$. The underlying algebra will be the free unital algebra $\mathcal{A}=\langle a_1, \hdots, a_m\rangle$, freely generated by elements which each are equipped with a {\bf grading} $|a_i| \in \ZZ/\ZZ\mu$. In addition, we will also require each generator to have an associated {\bf action} $\ell(a_i) \in \RR_{>0}$. The grading and the action are both extended to all monomials in the above generators of $\mathcal{A}$ using the formulae
\begin{eqnarray*}
& & |r|  = 0, \:\: r \in \kk; \\
& & |a_{i_1}\cdot \hdots \cdot a_{i_k}|=|a_{i_1}|+\hdots+|a_{i_k}| \in \ZZ/\ZZ\mu,
\end{eqnarray*}
together with
\begin{eqnarray*}
& & \ell(0) = -\infty; \\
& & \ell(r)  = 0, \:\: r \in \kk \setminus \{0\};\\
& & \ell(a_{i_1}\cdot \hdots \cdot a_{i_k})=\ell(a_{i_1})+\hdots+\ell(a_{i_k}) \in \RR_{>0}.
\end{eqnarray*}
We will be writing $\mathcal{A}_i \subset \mathcal{A}$ for the vector space over $\kk$ spanned by the monomials of degree $i \in \ZZ/\ZZ\mu$. The differential $\partial \colon \mathcal{A}_\bullet \to \mathcal{A}_{\bullet-1}$ is $\kk$-linear of degree $-1$, satisfies $\partial^2$, $\partial|_\kk =0$, together with the graded Leibniz rule
\begin{equation*}
\label{eq:leibniz}
\partial(\mathbf{a} \mathbf{b})=\partial(\mathbf{a}) \mathbf{b} +(-1)^{|\mathbf{a}|}\mathbf{a}\partial(\mathbf{b}),
\end{equation*}
for any monomials $\mathbf{a},\mathbf{b} \in \mathcal{A}$. Finally, we will also assume that:
\begin{enumerate}[label=(F):\:, ref=(F)]
\item \label{F} For any generator $a_i$, $i=1,\hdots,m$, the action of a monomial appearing in $\partial(a_i)$ with a non-zero coefficient is strictly less than $\ell(a_i)$.
\end{enumerate}

\begin{remark} For the Chekanov-Eliashberg algebra of a Legendrian submanifold $\Lambda \subset (P \times \RR,dz+\theta)$, the generators $a_1,\hdots,a_m$ are given by the so-called Reeb chords on $\Lambda$ (which here are assumed to be generic). Recall that a Reeb chord is an integral curve of $\partial_z$ having both endpoints on $\Lambda$. The grading is induced by the Conley-Zehnder index associated to a Reeb chord, which is well-defined modulo the Maslov number $\mu \in \ZZ$ of $\Lambda$, while the action simply is given by the length
\[ \ell(a_i) := \int_{a_i} (dz+\theta) \in \RR_{>0} \]
of the Reeb chord. We refer to \cite{NonIsoLeg} and \cite{ContHomR} for more details. The fact that the differential is of degree $-1$ and satisfies \ref{F} (i.e.~that it decreases the action) follows from the fact that it is defined by a count of rigid pseudoholomorphic polygons in $P$ having boundary on the canonical projection $\Pi_{\OP{Lag}}(\Lambda) \subset P$ of $P$ and corners at the double-points. Recall that the latter double-points are in a natural bijective correspondence with the Reeb chords. More precisely, the degree follows from the dimension formula for the moduli space of pseudo-holomorphic polygons expressed in terms of the Conley-Zehnder indices, while property \ref{F} follows from the formula for their symplectic areas, which necessarily is positive.
\end{remark}

\subsection{Stable-tame isomorphism}
\label{sec:stabletame}
For the applications that we have in mind, the strongest known invariant that can be extracted from the DGA is its stable-tame isomorphism class. We here proceed to give a definition. An {\bf elementary automorphism} of a semifree algebra $\mathcal{A}=\langle a_1,\hdots,a_m \rangle$ with a choice of generators is a unital algebra isomorphism of the form
\[ a_i \mapsto \begin{cases} a_i, & i \neq i_0, \\
ra_{i_0} + A, & i = i_0,
\end{cases}\]
for some $1 \le i_0 \le m$, and elements
\begin{align*}
& r \in \kk \setminus \{0\},\\
& A \in \langle a_1,\hdots,a_{i_0-1},a_{i_0+1},\hdots,a_m \rangle.
\end{align*}
A {\bf tame isomorphism} of two DGAs is one which can be decomposed as a sequence of elementary automorphisms, after first identifying their sets of generators. As a side note, it seems to be unknown whether there are DGA isomorphisms which are not tame.

Consider the {\bf stabilisation in degree $i \in \ZZ / \ZZ\mu$}, which is the DGA $(\mathcal{S}_i,\partial_{\mathcal{S}_i})$ with underlying algebra $\mathcal{S}_i:=\langle a,b \rangle$, grading given by $|a|=i$ and $|b|=i+1$, and differential defined by $\partial_{\mathcal{S}_i}(b)=a$. Two DGAs are {\bf stable-tame isomorphic} if they become tame isomorphic after taking free products with a number of stabilisations $(\mathcal{S}_i,\partial_{\mathcal{S}_i})$ for different $i \in \ZZ/\ZZ\mu$. 

By \cite[Section 3.1]{Computable} it follows that, up to algebra isomorphism and taking the free product with a free algebra, the characteristic algebra of a DGA is invariant under stable-tame isomorphisms of the DGA. (In fact, tameness is irrelevant here.)

\subsection{Degree functions and the weak algorithm}
\label{sec:weak}
Here we recall some techniques used in the study of free algebras. The weak algorithm is a generalisation of Euclid's division algorithm to the noncommutative setting due to P. Cohn. We refer to \cite[Section 2]{FreeRings} for an introduction. Below follows an overview of the results needed.

Let $\mathcal{A}=\langle a_1,\hdots,a_m\rangle$ be the free algebra with a choice of $m$ generators over $\kk$. Recall the definition of a {\bf degree function}
\[ \nu \colon \mathcal{A} \to \NN \cup \{ -\infty \}\]
given in \cite[Section 2.4]{FreeRings}, which is a function satisfying the following properties:
\begin{enumerate}
\item $\nu(x) \ge 0$ whenever $x \neq 0$, $\nu(0)=-\infty$;
\item $\nu(x-y) \le \max\{\nu(x),\nu(y)\}$;
\item $\nu(xy) = \nu(x) +\nu(y)$; and
\item $\nu(x)=0$ if and only if $x \in \kk \setminus \{0\}$.
\end{enumerate}
As an example, the usual degree $\nu_0$ of an element of $\mathcal{A}$, considered as a noncommutative polynomial in the indeterminates $a_1,\hdots,a_m$, clearly satisfies the above properties.

Let $\{a_{i,1}\cdot\hdots\cdot a_{i,{j_i}}\}_{i \in \NN}$, be the infinite $\kk$-basis of monomials in the generators $a_1,\hdots,a_m$ ordered by, say, the lexicographical order. The {\bf formal degree} of an element $x \in \mathcal{A}$ induced by $\nu$ and the above set of generators is given by
\begin{gather*}
x=\sum_{i \in \NN} r_i a_{i,1}\cdot\hdots\cdot a_{i,{j_i}} \in \mathcal{A},\\
\nu_{\OP{for}}(x):=\max_i\nu(r_ia_{i,1}\cdot \hdots \cdot a_{i,{j_i}}),
\end{gather*}
where the $r_i \in \kk$ above are uniquely determined and vanishing for all but finitely many values of $\NN$. Observe that $\nu=\nu_{\OP{for}}$ holds by definition in the case $\nu=\nu_0$.

A family $\{x_1,\hdots,x_n\} \subset \mathcal{A}$ of elements is called {\bf (left) $\nu$-dependent} given that there exist elements $\{y_1,\hdots,y_n\} \subset \mathcal{A}$ for which
\[ \nu\left(\sum_{i=1}^n y_i x_i\right) < \max_{i=1,\hdots,n} \nu(y_ix_i). \]
Moreover, we call such a family {\bf (left) $\nu$-independent} if it is not (left) $\nu$-dependent. An element $x \in \mathcal{A}$ satisfying the relation
\[ \nu\left( x-\sum_{i=1}^n y_i x_i \right) < \nu(x), \:\: \nu(y_ix_i) \le \nu(x), \: i=1,\hdots,n, \]
is said to be {\bf (left) $\nu$-dependent on the family $\{x_1,\hdots,x_n\}$}.

\begin{definition} The degree-function $\nu$ on $\mathcal{A}$ is said to satisfy the {\bf (left) weak algorithm} if any left $\nu$-dependent family $\{x_1,\hdots,x_n\} \subset \mathcal{A}$ of elements satisfying
\[ \nu(x_1) \le \hdots \le \nu(x_n) \]
has the property that $x_i$ is left $\nu$-dependent on $\{ x_1,\hdots,x_{i-1}\}$ for some $1 \le i \le n$.
\end{definition}
The main result needed from the theory of noncommutative rings is the following theorem.
\begin{theorem}[Theorems 2.5.1 (\S 2.5) and 2.4.6 (\S 2.4) in \cite{FreeRings}]
\label{thm:fir}
Assume that we are given a degree function as above which coincides with the induced formal degree, i.e.~that $\nu=\nu_{\OP{for}}$ holds. Then:
\begin{enumerate}
\item The degree-function $\nu$ satisfies the weak algorithm.
\item Any finitely generated left ideal $\mathfrak{a} =\mathcal{A}x_1+\hdots+\mathcal{A}x_m$ can be freely generated by a $\nu$-independent family $\{b_1,\hdots,b_n\} \subset \mathcal{A}$ of generators.
\end{enumerate}
\end{theorem}
Strictly speaking, we only rely on part (1) of the above theorem, while part (2) is stated for completeness. On the other hand, note that the crucial result Proposition \ref{prp:main} proven below is a refinement of the latter statement to a special setting needed here.

\subsection{A degree-function induced by the action filtration}
\label{sec:action}
Assume that we are given a DGA $(\mathcal{A}=\langle a_1,\hdots,a_m\rangle,\partial)$ with an action filtration $\ell$ taking values in $\RR_{\ge 0} \cup \{-\infty\}$ defined on the monomials. Again, consider the basis $\{a_{i,1}\cdot\hdots\cdot a_{i,{j_i}}\}_{i \in \NN}$ consisting of the monomials. We can extend the action to arbitrary elements by
\[ \ell\left(\sum_i r_ia_{i,1}\cdot \hdots \cdot a_{i,j_i}\right):=\max_i\ell(r_ia_{i,1}\cdot \hdots \cdot a_{i,j_i}), \]
where $r_i \in \kk \setminus \{0\}$ are non-zero for finitely many values of $i \in \NN$.

For a DGA with action filtration, we approximate the value of $\ell(a_i)$ for each generator, $i=1,\hdots,m$, by a rational number in $\QQ_{>0}$. In this way we can always assume that the action is $\QQ$-valued, i.e.~that we have a map $\widetilde{\ell} \colon \mathcal{A} \to \QQ_{ \ge 0} \cup \{-\infty\}$. Finally, we set
\[ \nu \colon \mathcal{A} \to \NN \cup \{-\infty\} \]
to be the action defined as above, but where we have assigned the value $N\widetilde{\ell}(a_i) \in \NN_{>0}$ to the generator $a_i$, where $N \in \NN_{>0}$ is the greatest common denominator of the actions $\widetilde{\ell}(a_1), \hdots, \widetilde{\ell}(a_m) \in \QQ_{ \ge 0}$ of the generators.
\begin{lemma}
The map $\nu$ satisfies the assumptions of a degree-function, as defined in Section \ref{sec:weak} above, for which $\nu=\nu_{\OP{for}}$ moreover is satisfied.
\end{lemma}
This degree-function is moreover compatible with the differential in the following sense.
\begin{lemma}
\label{lma:bdyaction}
Suppose that $\widetilde{\ell}$ was constructed so that $|\ell(a_i)-\widetilde{\ell}(a_i)|>0$ is sufficiently small for each $i=1,\hdots,m$. It then follows that, for any non-zero $x \in \mathcal{A}$, the inequality
\[ \nu(\partial(x)) < \nu(x)\]
holds for the degree-function $\nu \colon \mathcal{A} \to \NN \cup \{ -\infty \}$ obtained from $\widetilde{\ell}$ as above.
\end{lemma}
\begin{proof}
For a generator $a_k$, take a monomial $a_{i,1}\cdot \hdots \cdot a_{i,j_i}$ in the generators arising in the expression $\partial(a_k)$ with a non-zero coefficient. Since
\[ \ell(a_k) -(\ell(a_{i,1}) + \hdots + \ell(a_{i,j_i})) > 0,\]
holds by assumption \ref{F}, we conclude that
\[ \widetilde{\ell}(a_k) -(\widetilde{\ell}(a_{i,1}) + \hdots + \widetilde{\ell}(a_{i,j_i})) > 0\]
holds for the action defined above as well, given that $|\ell(a_i)-\widetilde{\ell}(a_i)|>0$ is sufficiently small for each $i=1,\hdots,m$. Since
\[ \nu(a_k) -(\nu(a_{i,1}) + \hdots + \nu(a_{i,j_i}))=N\widetilde{\ell}(a_i) -(N\widetilde{\ell}(b_1) + \hdots + N\widetilde{\ell}(b_n)) > 0\]
also is satisfied, the statement now follows. To that end, using the Leibniz rule, it suffices to obtain inequalities of the above form for the finite monomials arising with non-zero coefficients in the expressions $\partial(a_i)$, $i=1,\hdots,m$.
\end{proof}

\section{The proof of Theorem \ref{thm:main}}
Consider a cycle $x \in (\mathcal{A},\partial)$ which vanishes inside the characteristic algebra, i.e.~$[x]=0 \in \mathcal{C}$. The goal is showing that $x=\partial(y)$ for some element $y \in \mathcal{A}$. By definition, there exist elements $u_1,v_1,w_n,\hdots,u_n,v_n,w_n \in \mathcal{A}$ for which
\[ x=u_1 \partial(v_1)w_1+\hdots+u_n \partial(v_n)w_n.\]
Using the Leibniz rule, we may now write
\[ x = u_1\partial(v_1w_1)+\hdots+u_n\partial(v_nw_n)-((-1)^{|v_1|}u_1v_1\partial(w_1)+\hdots+(-1)^{|v_n|}u_nv_n\partial(w_n)).\]
In particular, $x$ is contained inside the left-ideal
\[ \mathfrak{a}:= \mathcal{A}\partial(v_1w_1)+\hdots+\mathcal{A}\partial(v_nw_n)+\mathcal{A}\partial(w_1)+\hdots+\mathcal{A}\partial(w_n)\]
generated by boundaries. By Proposition \ref{prp:main} below, we conclude that this left-ideal is \emph{freely} generated by a set of boundaries $\partial(y_1),\hdots,\partial(y_m) \in \mathcal{A}$. Writing 
\[ x = x_1 \partial(y_1)+\hdots+x_m\partial(y_m),\]
the fact that $x$ is closed together with the Leibniz rule implies that
\[ \partial(x)=\partial(x_1)\partial(y_1)+\hdots+\partial(x_m)\partial(y_m)=0.\]
Since the above set of generators is free, we must have $\partial(x_i)=0$ for each $i=1,\hdots,m$. We have thus shown that
\[x=\partial((-1)^{|x_1|}x_1 y_1+\hdots+(-1)^{|x_m|}x_m y_m),\]
which finishes the claim.

What is left is showing the following proposition, which is the core of the argument.
\begin{proposition}
\label{prp:main}
Any left ideal $\mathfrak{a} \subset \mathcal{A}$ of the form
\[\mathfrak{a}=\mathcal{A}\partial(x_1)+\hdots+\mathcal{A}\partial(x_n)\]
in a DGA $(\mathcal{A},\partial)$ with action filtration can be freely generated by a family $\partial(y_1),\hdots,\partial(y_m)$ of boundaries.
\end{proposition}

\subsection{The proof of Proposition \ref{prp:main}}
In the following, we let $\nu$ be the degree-function induced by the action $\ell$, as constructed in Section \ref{sec:action} above. The fact that the differential $\partial$ strictly decreases the degree $\nu$ will turn out to be crucial; see Lemma \ref{lma:bdyaction}.

First observe that, for each $C \ge 0$, the $\kk$-subspace
\[ \mathcal{A}^C:= \nu^{-1}([-\infty,C]) \subset \mathcal{A}\]
is finite-dimensional. We write
\[ M:=\max_{i=1,\hdots,n}\nu(\partial(x_i))\]
and also
\[ \mathfrak{b}^C := \mathfrak{b} \cap \nu^{-1}([-\infty,C]), \:\: C\ge 0,\]
for any subset $\mathfrak{b} \subset \mathcal{A}$. Using this notation, it follows that $\partial(\mathcal{A}) \cap \mathfrak{a}^M$ generates $\mathfrak{a}$ as a left-ideal.

Given an ideal
\[ \mathfrak{a}_{n_0}=\mathcal{A}\partial(y_1)+\hdots+\mathcal{A}\partial(y_{n_0}) \subset \mathfrak{a}\]
contained inside $\mathfrak{a}$ which is generated by the boundaries $\partial(y_i)$, $i=1,\hdots,n_0$, we define the following two properties:
\begin{enumerate}[label=(\Alph*):\:, ref=(\Alph*)]
\item \label{A} The left ideal $\mathfrak{a}_{n_0}$ is freely generated by a $\nu$-independent family $b_1,\hdots,b_{n_0}$ satisfying
\[ \nu(b_1) \le \hdots \le \nu(b_{n_0})  \le M,\]
where there moreover are elements $u^i_j \in \mathcal{A}$ for which
\begin{gather*}
b_i = \partial(y_i)-(u^i_1b_1+\hdots+u^i_{i-1}b_{i-1}),\\
\nu(b_i) \le \nu(\partial(y_i)), \nonumber \\
\nu(u^i_jb_j) \le \nu(\partial(y_i)), \nonumber
\end{gather*}
is satisfied for each $i=1,\hdots,n_0$.
\item \label{B} For any boundary $\partial(y) \in \mathfrak{a} \setminus \mathfrak{a}_{n_0}$ we have
\[\nu(\partial(y)+x) \ge \nu(b_{n_0})\]
for each element $x \in \mathfrak{a}_{n_0}$.
\end{enumerate}
From \ref{A} we easily deduce:
\begin{enumerate}[label=(\Alph*'):\:, ref=(\Alph*')]
\item The boundaries $\partial(y_1),\hdots,\partial(y_{n_0})$ form a free generating set for the left-ideal $\mathfrak{a}_{n_0}$. In particular, $\partial(\mathfrak{a}_{n_0}) \subset \mathfrak{a}_{n_0}$. \label{A'}
\end{enumerate}
To see \ref{A'} one can argue by induction. Namely, given that
\[ v_1 \partial(y_1) + \hdots + v_{i-1}\partial(y_{i-1})+v_i\partial(y_{i}) =0 \]
is satisfied, Property \ref{A} implies the existence of elements $v_1',\hdots,v_{i-1}' \in \mathcal{A}$ for which
\[ v_1' b_1 + \hdots + v_{i-1}' b_{i-1}+v_{i}b_{i} =0,\]
and hence $v_i=0$ follows.

The proposition will be proven by the following induction argument. We start by choosing a non-zero boundary $\partial(y_1) \in \mathfrak{a}$ satisfying the property that, for any non-zero boundary $\partial(x) \in \mathfrak{a}$, we have
\begin{equation}
\label{eq:min1}
0 \le \nu(\partial(y_1)) \le \nu(\partial(x)).
\end{equation}
Such an element exists by the finite-dimensionality of $\mathfrak{a}^M$, together with the fact that $\partial(\mathcal{A}) \cap \mathfrak{a}^M$ generates $\mathfrak{a}$. We are now ready to establish the base of the induction.
\begin{lemma}[\bf The base case]
The left-ideal
\[ \mathfrak{a}_1 := \mathcal{A}\partial(y_1) \subset \mathfrak{a},\]
satisfies \ref{A} and \ref{B}.
\end{lemma}
\begin{proof}
Property \ref{A} is immediate with $b_1:=\partial(y_1)$, while \ref{B} is shown as follows. Assume by contradiction that there exists a boundary
\[\partial(y) \in \mathfrak{a} \setminus \mathfrak{a}_1\]
which satisfies
\begin{eqnarray}
\label{eq:b}
& & \nu(\partial(y)-v_1 \partial(y_1)) < \nu(b_1)=\nu(\partial(y_1)),
\end{eqnarray}
for some $v_1 \in \mathcal{A}$, and write $b:=\partial(y)-v_1 \partial(y_1)$. The identity
\[ \partial(b)=-\partial(v_1)\partial(y_1),\]
combined with Lemma \ref{lma:bdyaction} and Formula \eqref{eq:b} now implies that $\partial(v_1)=0$. Hence, we conclude that
\[ b=\partial(y-v_1y_1) \in \mathfrak{a} \setminus \mathfrak{a}_1+\mathfrak{a}_1=\mathfrak{a} \setminus \mathfrak{a}_1\]
which, since $\nu(b) < \nu(\partial(y_1)$, leads to a contradiction with Formula \eqref{eq:min1}.
\end{proof}
The following lemma provides us with the induction step.
\begin{lemma}[\bf The induction step]
\label{lem:step}
Whenever there is a left-ideal $\mathfrak{a}_{n_0} \subsetneq \mathfrak{a}$ satisfying \ref{A} and \ref{B}, there exists a boundary $\partial(y_{n_0+1}) \in \mathfrak{a} \setminus \mathfrak{a}_{n_0}$ for which
\[\mathfrak{a}_{n_0+1}:=\mathfrak{a}_{n_0}+\mathcal{A}\partial(y_{n_0+1})\]
again satisfies properties \ref{A}, \ref{B}.
\end{lemma}
We claim that the proposition follows from this induction step. Indeed, recall that $\mathfrak{a}^M \supset (\mathfrak{a}_N)^M$ is a finite-dimensional vector space. Property \ref{A} states that $\mathfrak{a}_N$ is freely generated by
\[ \nu(b_1) \le \nu(b_2) \le \hdots \le(b_N) \le M,\]
which implies that $\mathfrak{a}_N=\mathfrak{a}$ for some $N \ge 0$. The proposition now follows from \ref{A'}. What remains is proving the induction step.
\begin{proof}[Proof of Lemma \ref{lem:step}]
Since $\mathfrak{a}$ is generated by $\partial(\mathcal{A})\cap \mathfrak{a}^M$, the subset
\[(\mathfrak{a} \setminus \mathfrak{a}_{n_0})^M \subset \mathfrak{a}^M\]
must contain a boundary. Here we have used the assumption that $\mathfrak{a}_{n_0} \subsetneq \mathfrak{a}$.

The finite-dimensionality of $(\mathfrak{a}\setminus \mathfrak{a}_{n_0})^M$ implies that we may choose an element $\partial(y_{n_0+1}) \in (\mathfrak{a} \setminus \mathfrak{a}_{n_0})^M$ satisfying the following properties:
\begin{enumerate}[label=($*$), ref=($*$)]
\item Writing
\[ m:=\min_{v_i \in \mathcal{A}}\nu (\partial(y_{n_0+1})-(v_1 b_1 + \hdots + v_{n_0}b_{n_0})) \in \NN,\]
the number $m \ge 0$ is minimal in the sense that \label{min}
\begin{equation}
\label{eq:min}
\min_{v_i \in \mathcal{A}} \nu (\partial(y)-(v_1 b_1 + \hdots + v_{n_0}b_{n_0})) \ge m,
\end{equation}
is satisfied for any $\partial(y) \in \mathfrak{a} \setminus \mathfrak{a}_{n_0}$.
\end{enumerate}
By property \ref{B} of $\mathfrak{a}_{n_0}$ we conclude that the inequality
\begin{equation}
m \ge \nu(b_{n_0}) \label{eq:mgeq}
\end{equation}
is satisfied. Furthermore, from the $\nu$-independence of the family $b_1,\hdots,b_{n_0}$ we establish the following property. Given any $\partial(y) \in \mathfrak{a} \setminus \mathfrak{a}_{n_0}$ and $v_i \in \mathcal{A}$, $i=1,\hdots,n_0$, satisfying
\[\nu(\partial(y)-(v_1 b_1 + \hdots + v_{n_0}b_{n_0}))=m \le \nu(\partial(y_{n_0+1})),\]
the inequalities
\begin{eqnarray*}
\nu(v_ib_i) &\le& \nu(\partial(y))
\end{eqnarray*}
must be satisfied for each $i=1,\hdots,n_0$.

{\bf Property \ref{A} for $\mathfrak{a}_{n_0+1}$:} We fix $v_i \in \mathcal{A}$, $i=1,\hdots,n_0$, as above for which
\[\nu(\partial(y_{n_0+1})-(v_1 b_1 + \hdots + v_{n_0}b_{n_0})) = m \ge \nu(b_{n_0})\]
and write
\[ b_{n_0+1}:=\partial(y_{n_0+1})-(v_1 b_1 + \hdots + v_{n_0}b_{n_0}).\]
It suffices to show that $b_1,\hdots,b_{n_0},b_{n_0+1}$ is $\nu$-independent. By the weak algorithm, which holds by part (1) of Theorem \ref{thm:fir}, it thus suffices to show that $b_{n_0+1}$ is not $\nu$-dependent on the family $b_1,\hdots,b_{n_0}$. To that end, recall that the latter family is $\nu$-independent by assumption. 

Indeed, given elements $u_i \in \mathcal{A}$, $i=1,\hdots,n_0$, the inequality
\[\nu(b_{n_0+1}-(u_1 b_1 + \hdots + u_{n_0}b_{n_0})) \ge m =\nu(b_{n_0+1})\]
must hold since, otherwise, we would obtain
\[\nu(\partial(y_{n_0+1})-((v_1+u_1) b_1 + \hdots + (v_{n_0}+u_{n_0})b_{n_0}))<m,\]
which clearly is in contradiction the choice of $\partial(y_{n_0+1})$ in \ref{min}.

{\bf Property \ref{B} for $\mathfrak{a}_{n_0+1}$:} This property is shown using the previously established Property \ref{A}. By contradiction, take $\partial(y) \in \mathfrak{a} \setminus \mathfrak{a}_{n_0+1}$ and $v_1,\hdots,v_{n_0+1} \in \mathcal{A}$ satisfying
\[ \nu(\partial(y) -(v_1 b_1 + \hdots + v_{n_0}b_{n_0} +v_{n_0+1}b_{n_0+1}))<m=\nu(b_{n_0+1}).\]
It is immediate from the construction of $m \ge 0$ in \ref{min} that $v_{n_0+1} \neq 0$.

Property \ref{A} for $\mathfrak{a}_{n_0+1}$ shows that we can write
\begin{equation}
\label{x}
v_1 b_1 + \hdots + v_{n_0}b_{n_0} +v_{n_0+1}b_{n_0+1}=v_1' b_1 + \hdots + v_{n_0}'b_{n_0} +v_{n_0+1}\partial(y_{n_0+1}),
\end{equation}
for suitable elements $v_i' \in \mathcal{A}$, $i=1,\hdots,n_0$. After using Property \ref{A'} for $\mathfrak{a}_{n_0}$ together with the Leibniz rule, we compute
\[x:=\partial( v_1' b_1 + \hdots + v_{n_0}'b_{n_0} +v_{n_0+1}\partial(y_{n_0+1}))=v_1''b_1+ \hdots + v_{n_0}''b_{n_0}+\partial(v_{n_0+1})b_{n_0+1},\]
for suitable elements $v_i'' \in \mathcal{A}$, $i=1,\hdots,n_0$. Lemma \ref{lma:bdyaction} implies that $\nu(x) < m = \nu(b_{n_0+1})$. By the $\nu$-independence of $b_1,\hdots,b_{n_0+1}$ together with $\nu(b_{n_0+1}) = m$ it follows that $\partial(v_{n_0+1})=0$.

In conclusion, we can construct a boundary
\[\partial(y-v_{n_0+1}y_{n_0+1}) \in \mathfrak{a} \setminus \mathfrak{a}_{n_0+1}+\mathfrak{a}_{n_0+1}=\mathfrak{a} \setminus \mathfrak{a}_{n_0+1}\]
which by Formula \eqref{x} satisfies
\[ \nu(\partial(y-v_{n_0+1}y_{n_0+1})-(v_1' b_1 + \hdots + v_{n_0}'b_{n_0})) <m.\]
This is clearly in contradiction with the construction of $m\ge0$ in \ref{min}.
\end{proof}

\bibliographystyle{alphanum}
\bibliography{references}

\def\cprime{$'$} \def\cprime{$'$} \def\cprime{$'$}
\begin{thebibliography}{EES3}

\bibitem[Che]{DiffAlg}
Y.~Chekanov.
\newblock Differential algebra of {L}egendrian links.
\newblock {\em Invent. Math.}, 150(3):441--483, 2002.

\bibitem[Coh]{FreeRings}
P.~M. Cohn.
\newblock {\em Free ideal rings and localization in general rings}, volume~3 of
  {\em New Mathematical Monographs}.
\newblock Cambridge University Press, Cambridge, 2006.

\bibitem[DR]{MyCaps}
G.~Dimitroglou~Rizell.
\newblock Exact {L}agrangian caps and non-uniruled {L}agrangian submanifolds.
\newblock {\em Ark. Mat.}, 53(1):37--64, 2015.

\bibitem[DRG]{EstimNumbrReebChordLinReprCharAlg}
G.~Dimitroglou~Rizell and R.~Golovko.
\newblock Estimating the number of {R}eeb chords using a linear representation
  of the characteristic algebra.
\newblock {\em Algebr. Geom. Topol.}, 15(5):2887--2920, 2015.

\bibitem[EES1]{ContHomR}
T.~Ekholm, J.~Etnyre, and M.~Sullivan.
\newblock The contact homology of {L}egendrian submanifolds in
  {$\mathbf{R}^{2n+1}$}.
\newblock {\em J. Differential Geom.}, 71(2):177--305, 2005.

\bibitem[EES2]{NonIsoLeg}
T.~Ekholm, J.~Etnyre, and M.~Sullivan.
\newblock Non-isotopic {L}egendrian submanifolds in {$\mathbf{R}^{2n+1}$}.
\newblock {\em J. Differential Geom.}, 71(1):85--128, 2005.

\bibitem[EES3]{ContHomP}
T.~Ekholm, J.~Etnyre, and M.~Sullivan.
\newblock Legendrian contact homology in {$P\times\mathbf{R}$}.
\newblock {\em Trans. Amer. Math. Soc.}, 359(7):3301--3335 (electronic), 2007.

\bibitem[EGH]{IntroSFT}
Y.~Eliashberg, A.~Givental, and H.~Hofer.
\newblock Introduction to symplectic field theory.
\newblock {\em Geom. Funct. Anal.}, (Special Volume, Part II):560--673, 2000.
\newblock GAFA 2000 (Tel Aviv, 1999).

\bibitem[Ng]{Computable}
L.~Ng.
\newblock Computable {L}egendrian invariants.
\newblock {\em Topology}, 42(1):55--82, 2003.

\bibitem[Siv]{Sivek:ContactHomology}
S.~Sivek.
\newblock The contact homology of {L}egendrian knots with maximal
  {T}hurston-{B}ennequin invariant.
\newblock {\em J. Symplectic Geom.}, 11(2):167--178, 2013.

\bibitem[Sul]{Sullivan:Infinitesimal}
D.~Sullivan.
\newblock Infinitesimal computations in topology.
\newblock {\em Inst. Hautes \'Etudes Sci. Publ. Math.}, (47):269--331 (1978),
  1977.

\end{thebibliography}
\address{Georgios Dimitroglou Rizell\\
Centre for Mathematical Sciences\\
University of Cambridge\\
Wilberforce Road\\
Cambridge, CB3 0WB\\
United Kingdom\\
\email{g.dimitroglou@maths.cam.ac.uk}}

\end{document}